\documentclass[a4paper,12pt]{article}
\usepackage{amssymb}
\usepackage{amsmath}
\usepackage{amsfonts}
\usepackage{amsthm}
\usepackage{latexsym}

\newtheorem{theorem}{Theorem}[section]
\newtheorem{lemma}[theorem]{Lemma}
\newtheorem{remark}[theorem]{Remark}
\newtheorem{proposition}[theorem]{Proposition}

\numberwithin{equation}{section}

\newcommand{\C}{\mathbb C}
\newcommand{\R}{\mathbb R}

\newcommand{\Z}{\mathbb Z}

\def\A{\mathbf{A}}
\def\e{\varepsilon}
\def\H{\hat{H}^1}
\def\D{{\cal D}}
\def\ra{\rightarrow}
\def\wra{\rightharpoonup}

\begin{document}

\title{Magneto-static vortices in two dimensional Abelian gauge theories}
\author{ J. Bellazzini  \thanks{Dipartimento di Matematica Applicata,
Universit\`a di Pisa, Via F. Buonarroti 1/c, 56127 Pisa, ITALY.
e-mail: \texttt{j.bellazzini@ing.unipi.it, bonanno@mail.dm.unipi.it}}
\and C. Bonanno \addtocounter{footnote}{-1}\footnotemark
\and G. Siciliano \thanks{Dipartimento di Matematica, Universit\`a
di Bari, Via Orabona 4, 70125 Bari, ITALY. e-mail:
\texttt {siciliano@dm.uniba.it}
}}
\date{}
\maketitle
\abstract{We study the existence of vortices of the Klein-Gordon-Maxwell
equations in the two dimensional case. In particular we find sufficient conditions for the
existence of vortices in the magneto-static case, i.e when the electric potential $\phi=0$. 
This result, due to the lack of suitable embedding theorems  for the vector potential $\A$ is achieved with the help of a penalization method.}
\section{Introduction}

In the Abelian gauge theory the interaction between a matter
field $\psi$ obeying  the nonlinear Klein-Gordon equation and the electromagnetic field 
represented by the gauge potentials $(\A,\phi)$
is described  by considering the Lagrangian density (see e.g. \cite{Rub02}, \cite{Yan00}) 
\begin{equation}\label{lagg}
{\cal L}={\cal L}_{0}+{\cal L}_{1}-W(|\psi|)
\end{equation}
where $\psi: \R \times \R^{N} \to \C$ and
\begin{eqnarray}
{\cal L}_0 &=& \frac 12 \left [ |(\partial_t +i \phi)\psi|^2-|(\nabla-i A)\psi|^2 \right ] \nonumber \\
{\cal L}_1 &=& \frac 12 |\partial_t\A +\nabla \phi|^2-\frac 12 |\nabla \times A|^2 \nonumber
\end{eqnarray}
being $W$ is a suitable nonlinear term $W:\R^{+} \to \R$. Making the variation of the total action
\begin{equation}
{\cal S} (\psi,\phi,\A) =\int ({\cal L}_{0}+{\cal L}_{1}-W(|\psi|))\ dx dt
\end{equation}
with respect to $\psi,\phi,\A$ we have the following set of equations
\begin{eqnarray}
& (\partial_t+i\phi)^2\psi-(\nabla-i\A)^2\psi+W'(|\psi|)\frac{\psi}{|\psi|}=0 \label{prima} \\
& \nabla\cdot(\partial_t\A+\nabla\phi)=\left(\text{Im}\frac{\partial_t\psi}{\psi}
 +\phi\right)|\psi|^2 \label{seconda}\\
& \nabla\times\left(\nabla\times\A\right)+\partial_t\left(\partial_t\A+\nabla\phi\right)=
  \left(\text{Im}\frac{\nabla\psi}{\psi}-\A\right)|\psi|^2  \label{terza}&
\end{eqnarray}
which correspond to the Euler-Lagrange equations for (\ref{lagg}). We refer to
these equations as the Klein-Gordon-Maxwell (KGM) equations.

Many papers are concerned with the existence of stationary  solutions 
of (\ref{prima})-(\ref{terza}) in the static situation, i.e  functions of the following form
\begin{equation} \label{vortex}
	\psi(x,t)=u(x)e^{i(S(x)-\omega t)}\,,\,\,u\in\mathbb R^+\,,\,\,\omega\in\mathbb R,\,\,
	S\in\mathbb R/2\pi\mathbb Z.
\end{equation}
with the electromagnetic potentials satisfying
$$\partial_t\A=0\ \ \text{and}\ \ \partial_t\phi=0$$
In this case equations  (\ref{prima})-(\ref{terza}) become
\begin{eqnarray}
    &-\Delta u+\left[|\nabla S-\A|^2-(\phi-\omega)^2\right]u+W'(u)=0& \label{u}\\
  &-\nabla\cdot\left[(\nabla S-\A)u^2\right]=0&  \label{div}\\
    &-\Delta \phi=(\omega-\phi)u^2& \label{gauss}\\
    &\nabla\times\left(\nabla\times\A\right)=\left(\nabla S-\A\right)u^2.& \label{rot}
\end{eqnarray}
It is possible to have three types of stationary non-trivial solutions
\begin{itemize}
\item electro-static solutions: $\A=0$, $\phi \neq 0$
\item magneto-static solutions: $\A\neq 0$, $\phi = 0$
\item electromagneto-static solutions: $\A\neq 0$, $\phi \neq 0$
\end{itemize}
under suitable assumptions on the nonlinear term $W$.

If the stationary solution $\psi(x,t)=u(x)e^{i(S(x)-\omega t)}$ admits a phase that depends only on time, i.e $S(x)=0$, we call this solution a standing wave solution, whereas if $S(x)\neq 0$ we call this solution a vortex.

In the literature there exist results both for standing waves and vortices in the electro, magneto and  electromagneto-static case, see for instance the books \cite{FE81}, \cite{RA89} and the more recent papers \cite{BF}, \cite{BF02}, \cite{BF07} and \cite{DA04}. In particular, for what concerns the existence of vortices, the classical results of \cite{AB57} and \cite{NO73} are obtained in the two dimensional case with a double-well shaped function $W$ of the type $W(s)=( 1-s^2 )^2$, whereas in \cite{BF}
three dimensional vortices are studied with $W(s)=\frac 1 2 s^2-\frac{s^p}{p}$ with $2<p<6$.

In this paper we study two dimensional vortices in the magneto-static case, 
i.e for $\phi=0$. This problem has a physical relevance due to the fact that
two dimensional magneto-static vortices arise in superconductivity, see for instance \cite{Fe}. The assumption $\phi=0$ readily implies $\omega=0$, hence stationary solutions do not depend on time and have null angular momentum although they have non-vanishing magnetic momentum.

We consider solutions $\psi$ of equations (\ref{u})-(\ref{rot}) of the form (\ref{vortex}) with $\omega=0$ and $S(x)=k\theta(x)$ where $\theta$ is the angular function
$$
\theta(x)=\text{Im}\log(x_1+ix_2),\,\, x=(x_{1},x_{2}) \in \R^2\setminus\{0\}
$$
and $k \in \Z \setminus\{0\}$ is a constant. A solution $\psi$ with this choice of $S$ is a vortex and the constant $k$ is called the vorticity. Notice that the function $\theta$ and its gradient $\nabla\theta(x)=\left(\frac{x_2}{r^2},\frac{-x_1}{r^2},0\right)$ are $C^\infty$ in $\R^2\setminus \{0\}$ and $|\nabla\theta|=1/r.$

With this ansatz equations (\ref{u})-(\ref{rot}) reduce to
\begin{eqnarray}
  &-\Delta u+|k\nabla \theta -\A|^2u+W'(u)=0& \label{original1} \\
  &\nabla\times\left(\nabla\times\A\right)=\left(k\nabla \theta-\A\right)u^2.&
  \label{original2}
\end{eqnarray}
The existence of non-trivial solutions of (\ref{original1}) and (\ref{original2}) depends on the assumptions on the nonlinear term $W$. For example if $W'(s)s\ge0$ then one can prove that any solution $(u,\A)$ has necessarily $u\equiv 0$. We prove that under the following assumptions on $W$ there exists a solution with nontrivial $u$. 

We take the potential $W$ of the following type
$$
    W(s)=\frac{1}{2} s^2 - R(s)
$$
where $R:\R^{+} \to \R$ satisfies:
\begin{itemize}
    \item $R(0)=R'(0)=0$\,,
    \item $\exists\, c>0 \text{ and } p>2 \text{ such that } |R(s)|\le c s^p$\,,
    \item $s R'(s)\ge pR(s)>0\quad\text{for}\quad s>0$\,.
\end{itemize}

Before stating our main result, we make a short remark on notation. Our problem is defined in $\mathbb R^2$, however, to give sense to expressions like $\nabla \times \A$, vectors will be thought of as three-vectors with null third component and depending only on two variables $(x_{1},x_{2})$. In particular hereafter we use the notation $|\nabla\A|^2=\sum_{i,j=1}^{3}(\partial_i A_j)^2$.

The main result of the paper is the following.
\begin{theorem}  \label{Main}
Under the above conditions on the potential $W$ there exists a (non-trivial)
solution $(u_0,\A_0)$ of (\ref{original1}) and (\ref{original2}) in the
sense of distributions, where
\begin{itemize}
	\item $u_0$ is positive, radial and satisfies
   $$\int|\nabla u_0|^2\,dx+\int\, \left( 1+\frac{1}{r^2} \right)\, u_0^2 \, dx<+\infty, \ \ r^2=x_1^2+x_2^2\, ;$$
  \item $\A_0$ is divergence free and $\int|\nabla\A_0|^2\,dx<+\infty.$ 
\end{itemize}
\end{theorem}

This work has been inspirated by the recent work by Benci and Fortunato 
 \cite{BF}, in which the existence of three dimensional vortices for KGM 
in the electro, magneto and electromagneto-static case is proved under the same 
assumptions on $W$, by using a mountain pass argument in a suitable functional space. 

The two dimensional case, due to the lack of suitable embedding theorems 
concerning the vector potentials $\A$, shall be treated however with a different
approach. We cannot barely apply the same ideas of  \cite{BF} due to the fact that the same mountain pass argument cannot be used. In this paper we follow a penalization argument, finding solutions of the ``perturbed problem''
\begin{equation}        \tag{$P_\e$}   \label{Pee}
    \left\{
     \begin{array}{l}
     -\Delta u+|k\nabla \theta -\A|^2u+W'(u)=0 \vspace{0.1cm} \\
    \nabla\times\left(\nabla\times\A\right)+\e\A=\left(k\nabla\theta-\A\right)u^2
    \end{array}
    \right.
\end{equation}
for $\e\in(0,1)$. A solution of the initial problem (\ref{original1}) and (\ref{original2}) will then be obtained by taking the limit for $\e\ra0$ of the solutions $(u_\e,\A_\e)$ of (\ref{Pee}).

One of the advantages of the perturbed problem is that the space of vector potentials $\A$ can be chosen such that a mountain pass theorem can be applied to find ``weak'' solutions of the problem.

The paper is organized as follows: in Section \ref{sec:FunctionalFramework} we introduce all the functional spaces that will be used, and in Section \ref{A natural constraint} we introduce a natural constraint for the functional associated to (\ref{Pee}), that is a manifold on which the problem is more tractable. Finally, in Section \ref{sec:ProofOfTheorem} we prove the main theorem.

\section{Functional framework} \label{sec:FunctionalFramework}

In the following, unlike otherwise specified, all the integrals, norms and functional
spaces are intended on $\R^2$.

We denote by $\| \cdot \|_p$ the $L^p$ norm, $H^1$ is the usual Sobolev space with norm
$$
\|u\|_{H^1}^2=\int (|\nabla u|^2+u^2)\,dx
$$
and $\H$ is the weighted Sobolev space endowed with norm
$$
\|u\|_{\H}^2=\|u\|_{H^1}^2+\int\frac{u^2}{r^2}\,dx\,,\quad r^2=x_1^2+x_2^2.
$$
We denote the $L^p$ norm of a vector $\mathbf X$ as
$$
\| \mathbf X \|_p := \left\| \, (\mathbf X \, \mathbf X)^{\frac 1 2}\, \right\|_p
$$
where no symbol is used for the inner product between vectors. Using this notations,
$\|\A\|^2_{H^1}=\|\nabla\A\|^2_2+\|\A\|_2^2$ where $\|\nabla\A\|_2^{2} = \sum_{j} \|\nabla A_{j} \|_2^{2}$.

Let us define the space
$$
H=\H\times \left(H^1\right)^3
$$
with norm $\|(u,\A)\|_H^2=\|u\|^{2}_{\H} +\|\A\|_{H^1}^{2}$, and the functional on $H$
\begin{multline*}
    J_\e(u,\A)=\frac{1}{2}\int\left(|\nabla u|^2+|\nabla\times\A|^2\right)\,dx+
       \frac{1}{2}\int|k\nabla\theta-\A|^2u^2\,dx \\
    +\frac{\e}{2}\int |\A|^2\,dx+\int W(u)\,dx.
\end{multline*}
Straightforward computations show that $J_\e$ is well defined and $C^1$ on $H$ thanks to the growth conditions on $W$, and its Euler-Lagrange equations are (\ref{Pee}). Hence a critical point $(u,\A)$ of $J_\e$ in $H$ is a weak solutions of (\ref{Pee}), that is
\begin{align}
  &\int \left(\nabla u\nabla v+|k\nabla\theta-\A|^2uv+W'(u)v\right)dx=0\,,
     \forall\, v\in \H& \label{eq1} \\
  &\int\left((\nabla\times\A) \,(\nabla\times\mathbf V)+\e\A\mathbf V+
    (\A-k\nabla\theta)\mathbf V u^2\right)dx=0 \,, \forall\, \mathbf V\in (H^1)^3.& \label{eq2}
\end{align}
For details we refer to \cite{BF} where the case $N=3$ is treated. 

\begin{remark}
We can extend the potential $W$ to be defined on $\R$ by letting $R(s)=0$ for $s\le 0$. Using this extension one proves that if the couple $(u,\A)$ satisfies (\ref{eq1}) and (\ref{eq2}), then $u\ge0$ a.e., hence $u$ has some physical consistency. Indeed, denoting with $u^-(x)=\min\{u(x),0\}$
and taking $v=u^-$ in (\ref{eq1}), we have
$$
\int\left[|\nabla u^-|^2+|k\nabla\theta-\A|^2(u^-)^2+W'(u^-)u^-\right]\,dx=0.
$$
Since $W'(s)=s-R'(s)=s$ for $s\le0$ we have
$$
\int\left[|\nabla u^-|^2+|k\nabla\theta-\A|^2(u^-)^2+(u^-)^2\right]\,dx=0
$$
and then $u^-=0$ a.e.
\end{remark}

A difficulty that arises looking at vortices of KGM is that the space $\D$ of test functions is not contained in $\H$. Hence a weak solution of (\ref{Pee}), a priori, does not satisfies it
in the sense of distributions, specifically (\ref{eq1}) for $v\in\D$. Fortunately this circumstance does not happen. In Proposition \ref{fine} we show that a weak solution
satisfying (\ref{eq1}) and (\ref{eq2}) turns out to be a solution in the sense
of distributions. Hence we first find a weak solution of (\ref{Pee}) and then obtain also a solution in the sense of distributions.

\section{A natural constraint for $J_\e$} \label{A natural constraint}

To study the existence of critical points of the functional $J_{\e}$, we restrict ourselves to a submanifold of the space $H$. This is due to some difficulties. Although the introduction of the parameter $\e$ helps us to work in the familiar space $(H^1)^3$, the functional $J_\e$, contains a term, $\int|\nabla\times\A|^2\,dx$, which is not a Sobolev norm.

To overcome this problem, looking at the identity
\begin{equation} \label{id}	
	\int \left( |\nabla\times \mathbf X|^2+(\nabla\cdot \mathbf X)^2 \right) = \int |\nabla\mathbf X|^2
\end{equation}
for regular vectors $\mathbf X$ with compact support, it seems natural, if we want to deal with $|\nabla\A|^2$ in place of $|\nabla\times\A|^2$, to take the manifold of divergence free vector fields.

Moreover, by classical results on symmetric solutions of elliptic problems, we are naturally led to introduce a constraint also on $u$, considering only radial functions.

Hence we introduce a manifold $V\subset H$ such that
\begin{enumerate}
    \item [(a)]it is a ``natural constraint'' for $J_\e$, namely its constrained
               critical points on $V$ are critical points on $H$;
    \item [(b)] any $\A\in V$ is divergence free;
    \item [(c)] any $u$ in $V$ is radially symmetric. 
\end{enumerate}
To be more precise, define
$$
{\cal A}_0=\{b\nabla\theta: b\in C^\infty_0(\mathbb R^2\setminus\{0\})\text{ and radial} \}
$$
and let
$$
{\cal A}=\text{the closure of ${\cal A}_0$ in the $(H^1)^3$ norm}.
$$
Since we consider only radial functions $b(x)$, they only depend on $r=|x|=(x_1^2+x_2^2)^{1/2}$. Hence we simply write $b(r)$. Moreover, notice that if ${\cal X}$ is the closure in the norm
$$
\|f\|_*^2=\int_0^{+\infty} \frac{f^2(r)}{r}\,dr+\int_0^{+\infty} \frac{(f'(r))^2}{r}\,dr
$$
of $C^\infty_0(0,+\infty)$, then
$$
{\cal A}=\{ b\nabla\theta: b\in{\cal X}\}.
$$
Moreover any $b\in{\cal X}$ can be continuously extended to $0$ by setting $b(0)=0$ and it results $b(r)=\int_0^r b'(t)\,dt$.

Define also $\D_r=\{u\in\D: u=u(r)\}$ and
\begin{center}
$\H_r$=\, the closure of $\D_r$ in the $\H$ norm.
\end{center}
The natural manifold we consider is then defined by
\begin{equation} \label{V}
	V=\H_r\times{\cal A}
\end{equation}
with norm $\|(u,\A)\|_V=\|(u,\A)\|_H$ (see Section \ref{sec:FunctionalFramework}).
\begin{remark}
The manifold $V$ is closed and convex, hence it is weakly closed in $H$. This will be used in the next section.
\end{remark}

We summarise the main properties of $\A$ and the advantages to consider $V$. First, since we are now dealing with radial functions $u$, we recall the following result which is used in the computations.
\begin{theorem}[\cite{BdFP}] \label{compatti} 
The space $H^1_r(\mathbb R^2,\mathbb R)$ is compactly embedded in $L^s(\mathbb R^2,\mathbb R)$ for $s\in(2,+\infty).$
\end{theorem}
For what concerns the vectors $\A$, the identity (\ref{id}) and vector calculus imply that
\begin{lemma}
For $\mathbf A\in{\cal A}$ we have
\begin{itemize}
    \item [1)]$\int|\nabla\times\A|^2\,dx=\int|\nabla\A|^2\,dx$\,;
    \item [2)]$\nabla\times\left(\nabla\times\A\right)=-\Delta\A$\,.
\end{itemize}
\end{lemma}
On $V$ the functional $J_\e$ has the following form to which we refer 
hereafter
\begin{multline}
    J_\e(u,\A)=\frac{1}{2}\int\left(|\nabla u|^2+|\nabla\A|^2\right)\,dx+
    \frac{1}{2}\int|k\nabla\theta-\A|^2u^2\,dx  \\
    +\frac{\e}{2}\int |\A|^2\,dx+\int W(u)\,dx \label{Je}
\end{multline}
A critical point $(u_0,\A_0)$ of $J$ on $V$ satisfies:
\begin{align}
  &\int \left(\nabla u_0\nabla v+|k\nabla\theta-\A_0|^2uv+W'(u_0)v\right)dx=0
     \,,\forall\, v\in \H_r& \label{wsol1} \\
  &\int\left(\nabla\A_0 \cdot \nabla\mathbf V+\e\A_0\mathbf V+(\A_0-k\nabla\theta)\mathbf V u^2\right)dx=0
    \,, \forall\, \mathbf V\in {\cal A}.& \label{wsol2}
\end{align}
i.e. it is a weak solution in $V$ of
\begin{equation}        \tag{$P_\e$}   \label{Pe}
    \left\{
     \begin{array}{l}
     -\Delta u+|k\nabla \theta -\A|^2u+W'(u)=0 \vspace{0.1cm} \\
     -\Delta \A+\e\A=\left(k\nabla\theta-\A\right)u^2
    \end{array}
    \right.
\end{equation}
The manifold $V$ defined in (\ref{V}) is a natural constraint according to the following theorem.
\begin{theorem} \label{natural}
Assume that $(u_0,\A_0)$ is a critical point of $J_\e$ in $V$, i.e.
\begin{equation*} 
    d J_\e(u_0,\A_0)[v,\mathbf V]=0\quad \forall\,(v,\mathbf V)\in V.
\end{equation*}
Then
\begin{equation} \label{conv}
	d J_\e(u_n,\A_n)[v,\mathbf V]=0\quad \forall\,(v,\mathbf V)\in \H\times(H^1)^3.
\end{equation}
\end{theorem}
\begin{proof}
The result will be obtained making use of the Palais Principle of Symmetric Criticality \cite{Pa79}.

Let us first observe that $J_\e$ is invariant under the group action
$$
T_g:(u,\A)\mapsto(u\circ g,g^{-1}\circ \mathbf V\circ g)
$$
where $g\in O(2)$ is a rotation in $\R^2$. We compute the set of fixed points for this action. Clearly in the first variable, $u$, this set is nothing but $\H_r$
and
\begin{equation} \label{vincolo u}
	\partial_u J_\e(u_0,\A_0)[v]=0\ \ \text{for any } v\in\H_r.
\end{equation}
Moreover writing a generic vector $\mathbf V(x,y)$ as
$$
\mathbf V(x,y)=a(x_1,x_2)\mathbf{t}+b(x_1,x_2)\mathbf r
$$
where $\mathbf t=(x_2/r,-x_1/r)$ and $\mathbf r=(x_1/r,x_2/r)$, being as usual $r^2=x_1^2+x_2^2$, the requirement that $g^{-1}\circ \mathbf V\circ g=\mathbf V$ implies that the coefficients $a$ and $b$ are radial. Hence vectors of type $a(r)\mathbf t+b(r)\mathbf r$ are fixed by the action of $T_g$ on the second variable.

We claim that
\begin{equation} \label{vincolo A}
	\partial_\A J_\e(u_0,\A_0)[a(r)\mathbf t+b(r)\mathbf r]=0.
\end{equation}
Indeed, by assumption,
$$
\A_0=b_0\nabla\theta=\frac{b_0(r)}{r}\mathbf t
$$
and $\partial_\A J_\e(u_0,\A_0)[a(r)\mathbf t]=0$.
In order  to prove (\ref{vincolo A}) we have only to show that
$$
\partial_\A J_\e(u_0,\A_0)[b(r)\mathbf r]=0.
$$
Since $\nabla\times(b(r)\mathbf r)=0$ and the vectors $\mathbf t$ and $\mathbf r$ are orthogonal, we have
\begin{multline*}
	\partial_\A J_\e(u_0,\A_0)[b(r)\mathbf r]=\int \left( \nabla\times \A_0\right) \, \left( \nabla\times(b(r)\mathbf r) \right)\,dx
	\\+\int u^2(\A_0-k\nabla\theta)b(r)\mathbf r\,dx+\e\int\A_0\, b(r)\mathbf r\,dx=0
\end{multline*}
which proves the claim.

We conclude, by (\ref{vincolo u}) and (\ref{vincolo A}), that the couple $(u_0,\A_0)$ is a
critical point of $J_\e$ on the set of fixed points for the action of $T_g$ on $\H\times(H^1)^3.$ Hence the Palais Principle applies and we get (\ref{conv}).
\end{proof}

\section{Proof of Theorem \ref{Main}} \label{sec:ProofOfTheorem}

By the previous results, we are reduced to study the functional $J_\e$ defined in (\ref{Je}) on $V= \H_{r} \times {\cal A}$.

\subsection{Solution of the perturbed problem}

\begin{proposition}\label{wls}
The functional $J_\e$ is weakly lower semicontinuous on $V$.
\end{proposition}
\begin{proof}
Using (\ref{Je}) we can write
\begin{multline*}
J_\e(u,\A)=\frac{1}{2}\|u\|^{2}_{\H_r}+\frac{1}{2}\|\nabla\A\|^{2}_{2}+
  \frac{\e}{2}\|\A\|^{2}_{2}-\int R(u)\,dx \\
  +\frac{1}{2}\int |\A|^2 u^2\,dx-k\int\nabla\theta\A u^2\,dx,
\end{multline*}
hence it is sufficient to show that the last two terms are weakly continuous.

Let $(u_n,\A_{n})\wra (u,\A)$ in $V$ for a given $(u,\A)\in V$. Then the norms $\| (u_n,\A_{n}) \|_{H}$ are bounded. We prove that
\begin{equation}     \label{convdeb}
    \int |\A_n|^2 u_n^2\,dx\ra\int|\A|^2 u^2\,dx \quad  \quad
    \int \nabla\theta\, \A_n u_n^2\,dx\ra\int \nabla\theta\, \A u^2\,dx.
\end{equation}
To prove the first convergence, we write
$$
\left| \int \left(|\A_n|^2u_n^2 - |\A|^2 u^2 \right) \, dx \right| \le  a_n+b_n
$$
with 
$$
a_n= \int|\A_n|^2|u_n^2-u^2|\,dx \le \|\A_n\|^{2}_{4}\left(\int\left|u_n^2-u^2\right|^2\,dx\right)^{1/2}
$$
$$
b_n= \left| \int u^2 \left( |\A_n|^2 - |\A|^2 \right) \, dx \right|
$$
By the compactness result of Theorem \ref{compatti}, up to a sub-sequence, we can assume that
$$
u_n^2\ra u^2\;\,a.e.\quad \text{ and } \quad\|u_n^2\|_2\ra\|u^2\|_2\,,
$$
and by the classical Sobolev embedding $H^{1}(\R^{2}) \subset L^{q}(\R^{2})$ for all $q\in [2,+\infty)$, the norm $\| \A_{n} \|_{4}$ is bounded. Hence it follows that $a_n\ra 0$. For what concerns $b_{n}$, by applying again the Sobolev embedding it follows that the functions $|\A_{n}|^{2}$ are in $L^{p}(\R^{2})$ for all $p\in [1,+\infty)$ and with bounded norms. Hence in particular $|\A_{n}|^{2}$ are bounded in $L^{2}$. Hence, up to a sub-sequence they converge weakly in $L^{2}$ to $|\A|^{2}$. Again by Theorem \ref{compatti}, the function $u^{2}$ is in $L^{2}$, hence $b_n\ra 0$.

Analogously, to prove the second convergence in (\ref{convdeb})
$$
\left| \int \left( \nabla \theta \left(\A u^2-\A_n u_n^2 \right) \right) \, dx \right| \le a'_n+b'_n
$$
with 
$$
a'_n = \int \frac{1}{r} |\A_{n}| |u^2-u_n^2|\,dx
$$ 
$$
b'_n = \int \frac{1}{r} u^2\, |\A - \A_n| \,dx.
$$
Using the Schwarz inequality
\begin{eqnarray*}
b'_n &\le& \left(\int |\A-\A_n|^2 u^2 \, dx \right)^{1/2} \left(\int \frac{u^2}{r^2}\,dx\right)^{1/2}\\
&\le& \|u\|_{\H_r} \left(\int|\A-\A_n|^2 u^2\,dx\right)^{1/2}
\end{eqnarray*}
and we can apply the same argument as before to the functions $|\A-\A_n|^2$ to obtain, up to a sub-sequence, the weak convergence in $L^{2}$. Hence $b'_{n}$ is vanishing.

It remains to prove that $a'_n\ra0$. From this the second convergence in \eqref{convdeb} follows and the proof is completed.

It results $a'_n\le c'_n+d'_n$ where
$$c'_n=\int\frac{1}{r}|\A_{n}||u||u-u_n|\,dx \quad \quad d'_n=
  \int\frac{1}{r}|\A_{n}||u_n||u-u_n|\,dx.$$
Now we have
\begin{eqnarray*}
    c'_n&\le&\left(\int\left(\frac{|\A_{n}||u|}{r}\right)^{3/2}\,dx\right)^{2/3}
             \left(\int|u_n-u|^{3}\,dx\right)^{1/3}\\
         &\le& \left(\left[\int\left(|\A_{n}|^{3/2}\frac{|u|^{1/2}}{r^{1/2}}\right)^2\,dx\right]^{1/2}
          \left[\int\frac{u^2}{r^2}\,dx\right]^{1/2}\right)^{2/3}\|u_n-u\|_3\\
          &\le& \left(\int |\A_{n}|^{3}\frac{u}{r}\,dx\right)^{1/3}\|u\|_{\H_r}^{2/3}\|u_n-u\|_3\\
          &\le& \|\A_{n} \|_6 \|u\|_{\H_r}^{1/3}\|u\|_{\H_r}^{2/3}\|u_n-u\|_3\\
          &=&\|\A_{n} \|_6 \|u\|_{\H_r}\|u_n-u\|_3\ra0
\end{eqnarray*}
by Theorem \ref{compatti} and because the norms $\| \A_{n} \|_{H^{1}}$ are bounded. Similarly 
$$
d'_n\le \|\A_{n}\|_6\|u_n\|_{\H_r}\|u_n-u\|_3\ra0
$$ 
which proves that $a'_n\ra0.$
\end{proof}

The next proposition establishes a geometrical property of $J_\e$ which enables us to
deduce a sequence of ``quasi-solutions'' i.e. a Palais-Smale sequence (PS for short).
\begin{proposition}\label{MP}
The functional $J_\varepsilon$ has the Mountain Pass geometry on $V$.
\end{proposition}
\begin{proof}
By Mountain Pass geometry we mean that there exist two constants $\rho, \alpha>0$ and a point $(\bar u,\bar\A)$ with $\|(\bar u,\bar\A)\|_V>\rho$ such that
\begin{align}
  &J_\e(0,\mathbf 0)=0 \nonumber \\
  &J_\e(u,\A)\ge \alpha \quad\text{for}\quad\|(u,\A)\|_V=\rho ,&  \label{MP1} \\	
  &J_\e(\bar u,\bar\A)\le 0 ,&  \label{MP2}
\end{align}
see \cite{ar}. It is worth noticing that $\bar u$ can be chosen independently on $\e$.

Let us first compute
\begin{eqnarray*}
\int |k\nabla\theta-\A|^2 u^2\,dx  &\ge& \int \left(|\A|^2-\frac{2|k\A|}{r}+\frac{1}{r^2}\right)u^2\,dx\\
&  =  & \int\left[|\A|^2-2\left(|k\A|\sqrt{2}\frac{1}{r\sqrt{2}}\right)+\frac{1}{r^2}\right]u^2\,dx \\
& \ge & \int\left[|\A|^2-2|k\A|^2-\frac{1}{2r^2}+\frac{1}{r^2}\right]u^2\,dx \\
&  =  & (1-2k^2)\int|\A|^2 u^2\,dx+\frac{1}{2}\int\frac{u^2}{r^2}\,dx\\
& \ge & \frac{1}{2}\int\frac{u^2}{r^2}\,dx+(1-2k^2)\|\A\|^{2}_{6}\|u\|^{2}_{3} \\
& \ge & \frac{1}{2}\int\frac{u^2}{r^2}\,dx-\frac{2k^2-1}{2}\|\A\|^{4}_{6}-\frac{2k^2-1}{2}\|u\|^{4}_{3}\\
& \ge & \frac{1}{2}\int\frac{u^2}{r^2}\,dx-c_1\|\A\|^{4}_{H^1}-c_2\|u\|^{4}_{\H_r}.
\end{eqnarray*}
So we have
\begin{eqnarray*}
J_\e (u,\A) & \ge & \frac{1}{2}\int|\nabla u|^2\,dx+\frac{1}{2}\int \frac{u^2}{r^2}\,dx
                    -c_1\|\A\|^{4}_{H^1}-c_2\|u\|^{4}_{\H_r}\\
            &  +  & \frac{1}{2}\|\nabla\A\|^2_2+\frac{\e}{2}\int |\A|^2\,dx+\int W(u)\,dx\\
            & \ge & \|u\|^{2}_{\H_{r}} \left( \frac{1}{2} - c_2\|u\|^{2}_{\H_r} \right) + \|\A\|^{2}_{H^1} \left( \frac{\e}{2} - c_1\|\A\|^{2}_{H^1} \right)\\ 
            &  -   & \int R(u)\,dx.
\end{eqnarray*}
By the assumptions on $R(s)$ 
$$
\int|R(u)|\,dx\le c\|u\|_p^p\le c'\|u\|_{\H_r}^p
$$ 
and hence $J_\e$ has a strict local minimum in $\left(0,\mathbf{0}\right)$ and (\ref{MP1}) is satisfied.

Finally we notice that, by using again the assumptions on $R$, for any $u_0\in\H_r$ it holds
$$
\lim_{t\rightarrow+\infty} J_\e(t u_0,\mathbf 0)=-\infty.
$$
Concluding, there exists a point $(\bar{u},\mathbf{0})$ such that $J_\e(\bar{u},\mathbf{0})<0$, hence (\ref{MP2}).
\end{proof}

By the $C^{1}$ regularity of the functional $J_{\e}$ and Proposition \ref{MP}, applying a weak form of the Mountain Pass Theorem we deduce the existence of a PS sequence for $J_\e$ at some level $c_\e> 0$. That is there exists a sequence $(u_n,\A_n)\subset V$ such that
$$
J_\e(u_n,\A_n)\ra c_\e \quad\text{and}\quad dJ_{\e}(u_n,\A_n)\ra 0 \;\;\text{in}\;\; V'.
$$
It is understood that the sequence $(u_n,\A_n)$ also depend on $\e$, but for simplicity we omit this dependence here and in the next two results.

The following lemma is fundamental.
\begin{lemma}
Let $(u_n,\A_n)\subset V$ be a PS sequence for the functional $J_\e$ at level $c_\e$. Then it is bounded.
\end{lemma}
\begin{proof}
If $(u_n,\A_n)\subset V$ is a PS sequence, by definition
\begin{align}
    \partial_u J_\e(u_n,\A_n) [u_n]= \lambda_n[u_n] \label{ps1}\\
    J_\e(u_n,\A_n)=c_{\e,n}\ra c_\e \label{ps2}
\end{align}
where $\lambda_n\ra0$ in $(\H)'$.

Evaluating 
$$
J_\e(u_n,\A_n)-\frac{1}{p} \partial_u J_\e(u_n,\A_n)[u_n] =c_{\e,n}-\frac{1}{p} \lambda_n[u_n]
$$
we find
\begin{multline}\label{stima}
    \frac{p-2}{2p}\int|\nabla u_n|^2\,dx+\frac{p-2}{2p}\int|k\nabla\theta-\A_n|^2u_n^2\,dx
    +\frac{1}{2}\int|\nabla\A_n|^2\,dx\\+ \frac{\e}{2}\int|\A_n|^2\,dx
    +\int \left(W(u_n)-\frac{1}{p}W'(u_n)u_n\right)dx=c_{\e,n}-\frac{1}{p} \lambda_n[u_n].
\end{multline} 
Recalling the assumptions on $W$, we get
$$W(u_n)-\frac{1}{p}W'(u_n)u_n=\frac{p-2}{2p}u_n^2+\frac{1}{p}R'(u_n)u_n-R(u_n)
  \ge\frac{p-2}{2p}u_n^2$$
hence \eqref{stima} implies
\begin{multline}
    \frac{p-2}{2p}\int|\nabla u_n|^2\,dx+\frac{p-2}{2p}\int|k\nabla\theta-\A_n|^2u_n^2\,dx
    +\frac{1}{2}\int|\nabla\A_n|^2\,dx   \nonumber\\
    +\frac{\e}{2}\int|\A_n|^2\,dx +\frac{p-2}{2p}\int u_n^2\,dx\le c_{\e,n}
    -\frac{1}{p} \lambda_n[u_n]
\end{multline}
hence
\begin{equation} \label{pslimit}
    \frac{p-2}{2p}\|u_n\|_{H^1}^2+\frac{1}{2}\|\nabla\A_n\|_2^2+\frac{\e}{2}\|\A_n\|_2^2
    \le c_{\e,n}+\frac{1}{p}\|\lambda_n\|_{(\H)'}\|u_n\|_{H^1}.
\end{equation}
By \eqref{pslimit} we deduce that $\{\|u_n\|_{H^1}\}$ and $\{\|\A_n\|_{H^1}\}$ are bounded.
In particular there exists a constant $C>0$ such that $\int R(u_n)\,dx\le C.$

Finally we have
\begin{eqnarray*}
  M &\ge& J_\e(u_n,\A_n)\ge\frac{1}{2}\|u_n\|_{\H_r}^2-k\int\nabla\theta\A_n u_n^2\,dx-\int R(u_n)\,dx\\
    &\ge& \frac{1}{2}\|u_n\|_{\H_r}^2-|k|\int|\A_n|\frac{u_n^2}{r}\,dx-C\\
    &\ge& \frac{1}{2}\|u_n\|_{\H_r}^2-\frac{|k|}{2}\int\left(4|\A_n|^2+\frac{1}{4r^2}\right)u_n^2\,dx-C\\
    &\ge& \frac{1}{2}\|u_n\|_{\H_r}^2-2|k|\int|\A_n|^2u_n^2\,dx-\frac{1}{8}\int\frac{u_n^2}{r^2}\,dx-C\\
    &\ge& \frac{1}{2}\|u_n\|_{\H_r}^2-2|k|\left\||\A_n|^2\right\|_2\|u_n^2\|_2-\frac{1}{8}\|u_n\|_{\H_r}^2-C\\
    & = & \frac{3}{8}\|u_n\|_{\H_r}^2-2|k|\|\A_n\|_4^2\|u_n\|_4^2-C
 \end{eqnarray*}
which shows that $\{u_n\}$ is bounded in $\H_r$ since $\{\|\A_n\|_4\}$ and $\{\|u_n\|_4\}$ are bounded by Sobolev embedding theorems.
\end{proof}
The next step is to prove that any PS sequence is bounded away from zero.
\begin{proposition} \label{away}
If $(u_{n},\A_{n})$ is a PS sequence for $J_\e$ at level $c_\e>0$ then for some $c>0$ 
$$
\|u_n\|^p_p\ge c>0.
$$
\end{proposition}
\begin{proof}
Let $\{(u_n,\A_n)\}$ be a bounded PS sequence satisfying \eqref{ps1} and \eqref{ps2}.
Since $\{\|u_n\|_{\H_r}\}$ is bounded,
\begin{equation} \label{away1}
    \|\nabla u_n\|_{2}^2+\int|k\nabla\theta-\A_n|^2u_n^2\,dx+\int W'(u_n)u_n\,dx
    =\lambda_n[u_n]\ra0.
\end{equation}
Hence
\begin{equation} \label{away2}
\begin{array}{c}
\|\nabla u_n\|_{2}^2+\int|k\nabla\theta-\A_n|^2u_n^2\,dx+\frac{1}{2}\int u_n^2\,dx = \\[0.3cm]
= \lambda_n[u_n]+\int R'(u_n)u_n\,dx \le \lambda_n[u_n]+\|u_n\|_p^p 
\end{array}
\end{equation}
from which it follows
\begin{equation} \label{away4}
    \|u_n\|_{H^1_r}^2\le\lambda_n[u_n]+\|u_n\|_p^p.
\end{equation}
We argue by contradiction. If $\|u_n\|_p\ra0,$ using \eqref{away1} and \eqref{away4} we obtain
\begin{equation} \label{away6}
    \|u_n\|_{H^1_r}^2\ra0 \quad\text{and}\quad \int R(u_n)\,dx\ra0
\end{equation}
and coming back to \eqref{away2}
\begin{equation} \label{away8}
    \int|k\nabla\theta-\A_n|^2u_n^2\,dx\ra0.
\end{equation}
On the other hand since $\partial_{\A} J_{\e} (u_{n},\A_{n}) \ra 0$ it holds
\begin{equation*}
    -\Delta\A_n+\e\A_n-\left(k\nabla\theta-\A_n\right)u_n^2=\delta_n\ra0 \quad \text{in}\quad ((H^1)^3)'
\end{equation*}
and, since $\{\|\A_n\|_{H^1}\}$ is bounded,
\begin{equation} \label{away12}
    \|\nabla\A_n\|_2^2+\e\|\A_n\|_2^2-\int\left(k\nabla\theta-\A_n\right)\A_n u_n^2\,dx
    =\delta_n[\A_n]\ra0.
\end{equation}
Classical estimates give
$$
\begin{array}{c}
    \left|\int\left(k\nabla\theta-\A_n\right)\A_n u_n^2\,dx\right| \le \\[0.3cm]
    \le \left(\int|k\nabla\theta-\A_n|^2u_n^2\,dx\right)^{1/2}\left(\int|\A_n|^2 u_n^2\,dx\right)^{1/2} \ra 0
\end{array}
$$
by \eqref{away8} and since $\int|\A_n|^2 u_n^2\,dx$ is bounded by the Schwartz inequality. Therefore by \eqref{away12} we get
\begin{equation} \label{away14}
    \|\nabla\A_n\|_2^2+\e\|\A_n\|_2^2\ra0.
\end{equation}
Finally, by \eqref{away6}, \eqref{away8} and \eqref{away14}
\begin{multline*}
    J_\e(u_n,\A_n)=\frac{1}{2}\|u_n\|_{H^1_r}^2+\frac{1}{2}\int|k\nabla\theta-\A_n|^2 u_n^2\,dx \\
    +\frac{1}{2}\|\nabla\A_n\|_2^2+\frac{\e}{2}\|\A_n\|_2^2-\int R(u_n)\,dx\ra0.
\end{multline*}
This is a contradiction since $J_\e(u_n,\A_n)\ra c_\e>0.$
\end{proof}

By the previous results, for every $\e\in(0,1)$ there exists $(u_{n,\e},\A_{n,\e})$, a bounded PS sequence for $J_\e$ at level $c_\e$. So we can extract a weakly convergent sub-sequence, denoted again with $(u_{n,\e},\A_{n,\e})$,  to a certain $(u_\e,\A_\e)\in V$. We know that $u_\e\ne0$ (Proposition \ref{away}) and  $J_\e(u_\e,\A_\e)\le c_\e$ (Proposition \ref{wls}). We have proved that
\begin{eqnarray*}
    &u_{n,\e}\wra u_\e\ne 0 \quad\text{in}\quad \H_r& \\
    &\A_{n,\e}\wra \A_\e \quad\text{in}\quad {\cal A}&\\
    &J_\e(u_\e,\A_\e)\le c_\e.&
\end{eqnarray*}
As stated in the next proposition, the weak limit $(u_\e,\A_\e)$ is a solution of the perturbed problem with fixed $\e$.
\begin{proposition} \label{e-soluz}
The couple $(u_\e,\A_\e)$ is a weak solution of (\ref{Pe}), i.e. it satisfies (\ref{wsol1})
and (\ref{wsol2}).
\end{proposition}
\begin{proof}
Let $\e$ be fixed. Since $(u_{n,\e},\A_{n,\e})$ is a PS sequence for $J_\e$ we have
$$
-\Delta u_{n,\e}+|k\nabla\theta-\A_{n,\e}|^2u_{n,\e}+W'(u_{n,\e})= \lambda_{n,\e}\ra0 \quad\text{in}\quad(\H_r)'
$$
which evaluated on $v\in\H_r$ gives
\begin{multline}
    \int \nabla u_{n,\e}\nabla v\,dx+k^2\int \frac{u_{n,\e} v}{r^2}\,dx+\int u_{n,\e} v\,dx \nonumber\\
    -2k\int\nabla\theta\A_{n,\e} u_{n,\e} v\,dx-\int R(u_{n,\e})v\,dx= \lambda_{n,\e}[v]. \nonumber
\end{multline}
Applying the same arguments of the proof of Proposition \ref{wls}, letting $n\to \infty$ we find that $(u_\e,\A_\e)$ is a solution of the first equation in (\ref{Pe}), i.e. satisfies
(\ref{wsol1}) with $v\in\H_r$. Analogously
$$
-\Delta\A_{n,\e}+\e\A_{n,\e}-(k\nabla\theta-\A_{n,\e})u_{n,\e}^2= \delta_{n,\e}\ra0\quad\text{in}\quad((H^1)^3)'
$$
which evaluated on $\mathbf{V}\in {\cal A}$ and passing to the limit in $n$ gives
$$
\int \nabla\A_\e\cdot\nabla \mathbf{V}\,dx+\e\int\A_\e\mathbf{V}\,dx= \int (k\nabla\theta-\A_\e)\mathbf{V}u_\e^2\,dx
$$
so that $(u_\e,\A_\e)$ solves (\ref{wsol2}) with $\mathbf V\in {\cal A}$.
\end{proof}
\begin{remark}\label{Oss}
By Theorem \ref{natural} $(u_\e,\A_\e)$ satisfies also (\ref{eq1}) and (\ref{eq2}). 
\end{remark}

\subsection{...and now $\e\ra0$} \label{sec:NowERa0}

In this section all the limits are taken for $\e$ which tends to $0^+$.

As we have seen, for any $\e\in(0,1)$, $u_\e\ne0$. Actually we have the following
\begin{lemma}
There exists a positive constant, $C$ such that for every $\e\in(0,1)$
$$
0<C\le\|u_\e\|_{H^1_r}.
$$
\end{lemma}
\begin{proof}
Since every $(u_\e,\A_\e)$ satisfies (\ref{Pe}), we have
$$
\|u_\e\|_{H^1_r}^2+\int|k\nabla\theta-\A_\e|^2 u_\e^2\,dx-\int R'(u_\e)u_\e\,dx=0
$$
hence
$$
\|u_\e\|_{H^1_r}^2\le\int R'(u_\e)u_\e\,dx\le c\|u_\e\|_{H^1_r}^p
$$
which shows that $\{u_\e\}$ is bounded away from zero.
\end{proof}
We also need to know that the sequence $\{u_\e\}$ is bounded in $\H_r$. This is stated in the next Lemma. We first give some preliminary remarks.

Recalling the definition of the mountain pass level
$$
c_\e=\inf_{\gamma\in\Gamma}\max_{0\le t\le 1}J_\e(\gamma(t)),
$$
where $\Gamma=\{\gamma\in C([0,1],V):\gamma(0)=(0,\mathbf 0),\,\gamma(1)=(\bar{u},\mathbf{0})\}$ and $J_{\e}(\bar u, \mathbf 0) \le 0$ (see Proposition \ref{MP}), consider the path
$$
\gamma_0:t\in[0,1]\mapsto (t\bar u,\mathbf{0})\in V.
$$
Then
$$
c_\e\le\max_{0\le t\le 1}J_\e(\gamma_{0}(t))=\max_{0\le t\le 1}J_0(t\bar u,\mathbf{0})
$$
which says that $\{c_\e\}$ is bounded by some positive constant $K$ which does not depend on $\e.$

Also we know that
\begin{align*}
    \partial_u J(u_\e,\A_\e)=0\\
    J(u_\e,\A_\e)\le c_\e.
\end{align*}
Again evaluating $J(u_\e,\A_\e)-\frac{1}{p} \partial_u J(u_\e,\A_\e)[u_\e]\le c_\e$
we find
\begin{multline*}
    \frac{p-2}{2p}\|\nabla u_\e\|_2^2+\frac{p-2}{2p}\int|k\nabla\theta-\A_\e|^2 u^2_\e\,dx
    +\frac{1}{2}\|\nabla\A_\e\|_2^2 \\
    +\frac{\e}{2}\int|\A_\e|^2\,dx+\int \left(W(u_\e)-\frac{1}{p}W'(u_\e)u_\e\right)\,dx
    \le c_\e\le K.
\end{multline*}
Since by the assumptions 
$$
W(u_\e)-\frac{1}{p}W'(u_\e)u_\e\ge\frac{p-2}{2p}\,u_\e^2
$$
we find
\begin{equation} \label{limitatezza}
	\frac{p-2}{2p} \left( \|u_\e\|^2_{H^1_r}+ \int|k\nabla\theta-\A_\e|^2 u^2_\e\,dx \right)
	+\frac{1}{2}\|\nabla\A_\e\|_2^2+\frac{\e}{2}\int|\A_\e|^2\,dx\le K
\end{equation}
so that $\{u_\e\}$  is bounded in $H^1_r$.

Moreover by (\ref{limitatezza}) other information can be deduced.
\begin{lemma} \label{Lemma}
The following facts hold:
\begin{enumerate}
 \item $\{u_\e\}$ is bounded in $\H_r$,
 \item $\{\mathbf A_\e\}$ is bounded in $H^1_{loc}$,
 \item $\lim_{\e\ra0} \e\int\A_\e\mathbf V\,dx=0$ for any $\mathbf V\in (L^2)^3.$
\end{enumerate}
\end{lemma}
\begin{proof}
1. Since we have already proved that $\{u_\e\}$ is bounded in $H^1_r$,
it remains to prove the boundedness of $\int\frac{u_\e^2}{r^2}\,dx$. We write
$$
 \int\frac{u_\e^2}{r^2}\,dx=\int_0^1 \frac{u_\e^2}{r}\,dr+\int_1^{+\infty} \frac{u_\e^2}{r}\,dr
$$
and both integrals in the right hand side are uniformly bounded in $\e$, indeed
$$
 \int_1^{+\infty} \frac{u_\e^2}{r}\,dr\le \int_1^{+\infty} r u_\e^2 \,dr
 \le\int u_\e^2\,dx\le\|u_\e\|_{H^1_r}^2\le K\,,
$$ 
and by (\ref{limitatezza})
\begin{eqnarray*}
	\frac{2p}{p-2}K &\ge& \int|k\nabla\theta-\A_\e|^2u_\e^2\,dx =
                         \int|k-b_\e|^2\frac{u_\e^2}{r^2}\,dx\\
                  &\ge& \int_0^1|k-b_\e|^2\frac{u_\e^2}{r}\,dr\ge c
	                       \int_0^1\frac{u_\e^2}{r}\,dr
\end{eqnarray*}       
where the constant $c$ can be chosen independently on $\e$ since $k\neq0$ and $b_\e(0)=0$.

2. We have only to show that for any $\rho >0$
$$
\int_{B_\rho}|\A_\e|^2\,dx \quad\text{ is bounded independently on }\,\e
$$
where $B_\rho$ is the ball in $\mathbb R^2$ centered in $0$ and with radius $\rho$.

We have
\begin{eqnarray} \label{loc}
 K&\ge&\int|\nabla\A_\e|^2\,dx=\int|\nabla\times\A_\e|^2\,dx\nonumber\\
 &=& \int|\nabla\times(b_\e\nabla\theta)|^2\,dx=\int\frac{(b'_\e)^2}{r^2}\,dx\nonumber\\
 &=& \int_0^{+\infty} \frac{(b'_\e)^2}{r}\,dr.
\end{eqnarray}
Let us fix $\rho>0$. For $r\le \rho$, by the elementary inequality $b_\e(r)=\int_0^{r} b'_\e(t)\,dt\le \sqrt{r}(\int_0^r (b'_\e(t))^2\,dt)^{1/2}$ and (\ref{loc}) we find
\begin{eqnarray}\label{r}
(b_\e(r))^2
  &\le& r\int_0^r (b'_\e(t))^2\,dt \nonumber\\
  &\le& r^2 \int_0^r \frac{(b'_\e(t))^2}{t}\,dt \nonumber\\
  &\le& r^2 K.
\end{eqnarray}
Now we can evaluate $\int_{B_\rho}|\A_\e|^2\,dx$. It results
$$
\int_{B_\rho} |\A_\e|^2 \,dx =\int_0^\rho \frac{(b_\e)^2}{r}\,dr \le K \int_0^\rho r\,dr = \frac{1}{2} \rho^2 K
$$
in virtue of (\ref{r}).

3. By (\ref{limitatezza}) we deduce that $\{\sqrt{\e}\A_\e\}_{\e\in(0,1)}$ is bounded in $L^2$, so up to a sub-sequence, it weakly converges to a certain $\mathbf X$ in $(L^2)^3$, that is
$$
\int \sqrt{\e}\A_\e \mathbf V\,dx\ra\int\mathbf X\mathbf V\,dx\quad\forall\, \mathbf V\in (L^2)^3
$$
and hence the conclusion follows.
\end{proof}

As a consequence of Lemma \ref{Lemma}, we infer that there exists $(u_0,\A_0)\in\H_r\times{\cal A} = V$, such that as $\e\ra0$
\begin{align}
  &u_\e\wra u_0\,\quad\text{ in}\,\, \H_r,&\label{convdeb1}\\
	&\A_\e\wra\A_0\quad\text{in}\,\, H^1_{loc},&\label{convdeb2}
\end{align}
and so by Theorem \ref{compatti} and usual Sobolev embedding theorems
\begin{align}
  &u_\e\ra u_0\,\quad\text{ in}\,\,L^p\quad\,\text{for}\,\,\,2<p<+\infty,\label{uLp}&\\
	&\A_\e\ra\A_0\quad\text{in}\,\, L^p_{loc}\,\,\,\text{for}\,\,\,1\le p<+\infty,&\label{Lpforte}\\
	&\A_\e\ra\A_0\quad\text{a.e. in }\,\, \mathbb R^2.\label{ae}&
\end{align}

The proof of Theorem \ref{Main} is finished once we prove the following
\begin{proposition} \label{fine}
The couple $(u_0,\A_0)$  is a solution of (\ref{original1}) and (\ref{original2})
in the sense of distributions.
\end{proposition}
\begin{proof}
By Proposition \ref{e-soluz}, $(u_\e,\A_\e)$ are weak solutions of (\ref{Pe})
\begin{align*}
    -\Delta u_\e+|k\nabla\theta-\A_\e|^2 u_\e+W'(u_\e)=0,\\
    -\Delta\A_\e+\e\A_\e=(k\nabla\theta-\A_\e)u_\e^2,
\end{align*}
i.e. satisfy (\ref{wsol1}) and (\ref{wsol2}). Moreover by Theorem \ref{natural}
they satisfies (\ref{wsol1}) and (\ref{wsol2}) also with $v\in\H$ and $\mathbf V\in(H^1)^3$. In particular, for $v\in{\cal D}\left(\mathbb R^2\setminus\{0\}\right)$, (\ref{wsol1}) reads
\begin{multline*}
    \int\nabla u_\e\nabla v\,dx+k^2\int\frac{u_\e v}{r^2}\,dx+\int u_\e v\,dx+\int\A_\e^2 u_\e v\,dx\\
    -2k\int\nabla\theta\A_\e u_\e v\,dx-\int R'(u_\e)v\,dx=0.
\end{multline*}
By (\ref{convdeb1}) we have
\begin{equation} \label{prodscal}
\begin{array}{c}
  \int\nabla u_\e\nabla v\,dx+ k^2\int\frac{u_\e v}{r^2}\,dx+\int u_\e v\,dx \ra  \\[0.3cm]
   \ra \int\nabla u_0\nabla v\,dx+k^2\int\frac{u_0 v}{r^2}\,dx +\int u_0 v\,dx
\end{array}
\end{equation}
and it is clear that
\begin{equation} \label{R}
	\int R'(u_\e)v\,dx\ra\int R'(u_0)v\,dx.
\end{equation}
Moreover, if $B$ is a ball containing the support of $v$ we have
\begin{multline*}	
		\left|\int|\A_\e|^2u_\e v\,dx-\int|\A_0|^2u_0 v\,dx\right|\le\\
		\int_B|v||\A_\e|^2|u_\e-u|\,dx+\int_B|v|\left||\A_\e|^2-|\A_0|^2\right||u|\,dx
\end{multline*}
and by (\ref{uLp}), (\ref{Lpforte}) and  (\ref{ae})
$$
\int_B|v||\A_\e|^2|u_\e-u|\,dx\le c\|\A_\e\|_{L^{3}(B)}^2\|u_\e-u\|_3\ra0
$$
$$
\int_B|v|\left||\A_\e|^2-|\A_0|^2\right||u|\,dx\le c\left\||\A_\e|^2-|\A_0|^2\right\|_{L^2(B)}\|u\|_2\ra 0.
$$
This shows that
\begin{equation} \label{100}
	\int|\A_\e|^2u_\e v\,dx\ra\int|\A_0|^2u_\e v\,dx.
\end{equation}
Similarly,
\begin{multline*}
	\left|\int\nabla\theta\A_\e u_\e v\,dx-\int\nabla\theta\A_0 u_0 v\,dx\right|\le\\
	\int_B\left|\nabla\theta\A_\e u_\e v-\nabla\theta\A_\e u_0 v\right|\,dx+
	\int_B\left|\nabla\theta\A_\e u_0 v-\nabla\theta\A_0 u_0 v\right|\,dx
\end{multline*}
and it holds
\begin{eqnarray*}
	\int_B\frac{|v|}{r}|\A_\e||u_\e-u_0|\,dx\le c\|\A_\e\|_{L^{3/2}(B)}\|u_\e-u_0\|_{L^3(B)}\ra0,\\
	\int_B\frac{|v|}{r}|u_0||\A_\e-\A_0|\,dx\le c\|u_0\|_2\|\A_\e-\A_0\|_{L^2(B)}\ra0
\end{eqnarray*}
(notice that the function $v/r$ still belongs to ${\cal D}\left(\mathbb R^2\setminus\{0\}\right)$).
In other words
\begin{equation}\label{120}
	\int\nabla\theta\A_\e u_\e v\,dx\ra \int\nabla\theta\A_0 u_0 v\,dx.
\end{equation}
Putting together (\ref{prodscal}), (\ref{R}), (\ref{100}) and (\ref{120})
we infer that for any $v\in{\cal D}\left(\mathbb R^2\setminus\{0\}\right)$
\begin{equation}\label{130}
	\int\nabla u_0 \nabla v\,dx+\int|k\nabla\theta-\A_0|^2 u_0 v\,dx+\int W(u_0)v\,dx=0.
\end{equation}
To conclude that $(u_0,\A_0)$ is a solution of (\ref{original1}) in the sense of
distributions we need to show that (\ref{130}) is still true for $v\in{\cal D}.$ This is done by following an argument of \cite{BF} to which the reader is referred, here we sketch the main steps.

\medskip \emph{Step 1}
First, one defines a family of smooth and radial functions on $\mathbb R^2$
satisfying
\begin{itemize}
    \item $\chi_n(r)=1$\, for $r\ge2/ n,$
   \item $\chi_n(r)=0$\, for $r\le1/ n,$
    \item $|\chi_n(r)|\le1,$
    \item $|\nabla\chi_n(r)|\le2n,$
    \item $\chi_{n+1}(r)\ge\chi_n(r).$
\end{itemize}
It is not difficult to prove that if $\varphi\in H^1\cap L^\infty$ has bounded support then, possibly up to sub-sequences,
\begin{equation}\label{lemmaBF}
	\varphi_n:=\varphi\chi_n\wra\varphi \text{ in }H^1.
\end{equation}
Thank to these cut-off functions can be proved that $(u_0,\A_0)$ is a solution
of (\ref{eq1}) in the sense of distributions.

\medskip \emph{Step 2}
Now take $v\in{\cal D}$, and choose $\varphi_n=v^+\chi_n\in\H$ as test functions
in (\ref{eq1}). Observe that there exists a ball $B$ such that all the functions
$\varphi_n$ have support in $B.$ Then the proof of Theorem 8 of \cite{BF} can be adapted
here. Hence, taking  the limit in $n$ and making use of (\ref{lemmaBF})
(that in this case means $\varphi_n\wra v^+$ in $H^1$) we find
that (\ref{eq1}) is satisfied with $v^+$ as test functions.
Since the same is true for $v^-$, this yields that $(u_0,\A_0)$ solves in the sense
of distributions (\ref{eq1}), or equivalently (\ref{original1}).

\medskip We now prove that $(u_0,\A_0)$ is a solution of (\ref{original2}) in the sense of distributions. Certainly $(u_0,\A_0)$ satisfies also (\ref{eq2}) with $\mathbf V\in(\D)^3$.
We have to prove that $(u_0,\A_0)$ solves equation (\ref{original2}) in the sense
of distributions, or equivalently, since we are in the natural constraint, the equation
\begin{equation}\label{finale}
	-\Delta\mathbf A=\left(k\nabla\theta-\A\right)u^2
\end{equation}
 in the sense of distributions.

Therefore take $\mathbf V\in ({\cal D})^3$ and let $B$ be a ball containing the support of $\mathbf V$. We know that
$$
\int\nabla\A_\e\cdot\nabla\mathbf V\,dx+\e\int\A_\e\mathbf V\,dx
=\int(k\nabla\theta-\A_\e)\mathbf Vu_\e^2\,dx
$$
and we want to pass to the limit for $\e\ra0.$

We have
\begin{multline*}
\left|\int(k\nabla\theta-\A_\e)\mathbf Vu_\e^2\,dx-\int(k\nabla\theta-\A_0)\mathbf Vu_0^2\,dx\right|\le\\
\int_B u_\e^2|\mathbf V||\A_0-\A_\e|\,dx+\int_B|\mathbf V||k\nabla\theta-\A_0||u_\e^2-u_0^2|\,dx.
\end{multline*}
Now, again using (\ref{uLp}) and (\ref{Lpforte})
$$
\int_B u_\e^2|\mathbf V||\A_0-\A_\e|\,dx\le max|\mathbf V|\,\|u_\e\|^2_4\|\A_0-\A_\e\|_{L^2(B)}\ra0
$$
$$
\int_B|\mathbf V||k\nabla\theta-\A_0||u_\e^2-u_0^2|\,dx\le
  max|\mathbf V|\,\|k\nabla\theta-\A_0\|_{L^{3/2}(B)}\|u_\e^2-u_0^2\|_3\ra0
$$
so that
\begin{equation} \label{300}
	\int(k\nabla\theta-\A_\e)\mathbf Vu_\e^2\,dx\ra\int(k\nabla\theta-\A_0)\mathbf Vu_0^2\,dx.
\end{equation}
Moreover by (\ref{limitatezza}) there exists $\mathbf B\in L^2$ such that
$\nabla\A_\e\wra \mathbf B$ in $L^2$ and therefore the convergence is in the sense
of distributions, that is
$$
\nabla\A_\e\ra\mathbf B\quad\text{ in } \ \ {\cal D}' .
$$
On the other hand (\ref{convdeb2}) implies $\A_\e\ra\A$ in ${\cal D}'$ and then
$$
\nabla\A_\e\ra\nabla\A_0\ \  \text{in} \ \ {\cal D}'
$$
so necessarily $\nabla\A_0=\mathbf B\in L^2$. Finally, by virtue of (3) of Lemma \ref{Lemma}
\begin{equation} \label{350}
	\int\nabla\A_\e\cdot\nabla\mathbf V\,dx+\e\int\A_\e\mathbf V\,dx\ra\int\nabla\A_0\cdot\nabla\mathbf V\,dx.
\end{equation}
By (\ref{300}) and (\ref{350}) equation (\ref{finale}) is satisfied in the sense of distributions. Hence $(u_0,\A_0)$ satisfies (\ref{original2}) in the sense of distributions.
\end{proof}

\section{Acknowledgments}
The authors are supported by MIUR - PRIN2005 ``Metodi variazionali e topologici
nello studio di fenomeni non lineari''.\\ The authors are grateful to Vieri Benci for useful and stimulating discussions. The third author wishes also to thank the ``Dipartimento di Matematica Applicata\,--\,U. Dini'' of the University of Pisa, for the warm hospitality during the visiting period in which this work has been carried out.

\end{document}